\documentclass[letterpaper]{amsart}
\usepackage{amsfonts,amssymb,amscd,amsmath,enumerate,url,verbatim}
\usepackage[all]{xy}
\SelectTips{eu}{}
\xyoption{curve}
\CompileMatrices

\newcommand{\AQ}[4]{\operatorname{D}_{#1}(#2|#3,#4)}
\newcommand{\qci}{{q.c.i.\! }}

\newcommand{\lra}{\longrightarrow}

\newcommand{\ov}{\overline}

\newcommand{\dd}{\partial}
\newcommand{\fm}{{\mathfrak m}}
\newcommand{\ann}{\operatorname{ann}}
\newcommand{\fa}{{\mathfrak a}}

\newcommand{\fb}{{\mathfrak b}}
\newcommand{\fn}{{\mathfrak n}}
\newcommand{\fp}{{\mathfrak p}}

\newcommand{\bw}{{\mathsf\Lambda}}
\newcommand{\bd}{\boldsymbol}
\newcommand{\cls}{\operatorname{cls}}

\newcommand{\Ima}{\operatorname{Im}}

\newcommand{\Ker}{\operatorname{Ker}}
\newcommand{\pd}{\operatorname{pd}}

\newcommand{\grade}{\operatorname{grade}}
\newcommand{\cidim}{\operatorname{CI-dim}}
\newcommand{\cx}{\operatorname{cx}}

\newcommand{\rank}{\operatorname{rank}}
\newcommand{\HH}{\operatorname{H}}

\newcommand{\Tor}{\operatorname{Tor}}
\newcommand{\Ext}{\operatorname{Ext}}

\newcommand{\Hom}{\operatorname{Hom}}
\newcommand{\Po}{\operatorname{P}}

\theoremstyle{plain}
\newtheorem{theorem}{Theorem}[section]

\newtheorem*{Theorem2}{Theorem 2}

\newtheorem*{Theorem1}{Theorem 1}
\newtheorem*{Theorem3}{Theorem 3}
\newtheorem*{Main Theorem}{Main Theorem}

\newtheorem{proposition}[theorem]{Proposition}
\newtheorem{observation}[theorem]{Observation}
\newtheorem{lemma}[theorem]{Lemma}
\newtheorem{corollary}[theorem]{Corollary}

{\Alph{theorem}}
{\Alph{theorem}}

\theoremstyle{definition}

\newtheorem{definition}[theorem]{Definition}

\newtheorem{chunk}[theorem]{}

\newtheorem{remark}[theorem]{Remark}

\theoremstyle{remark}
\newtheorem*{Acknowledgment}{\bf Acknowledgment}
\newtheorem{example}[theorem]{Example}

\newenvironment{bfchunk}{\begin{chunk}\textit}{\end{chunk}}

\newtheorem*{bfRemark}{\bf Remark}
\numberwithin{equation}{theorem}

\begin{document}
\title[Quasi-complete intersection ideals]
{Minimal quasi-complete intersection ideals}
\date{\today}

\author[A.~R.~Kustin]{Andrew R.~Kustin}
\address{Andrew R.~Kustin\\ Department of Mathematics\\ University of South Carolina\\\linebreak 
Columbia\\ SC 29208\\ U.S.A.} \email{kustin@math.sc.edu}

\author[L.~M.~\c{S}ega]{Liana M.~\c{S}ega}
\address{Liana M.~\c{S}ega\\ Department of Mathematics and Statistics\\
   University of Missouri\\ \linebreak Kansas City\\ MO 64110\\ U.S.A.}
     \email{segal@umkc.edu}

\author[A.~Vraciu]{Adela~Vraciu}

\address{Adela~Vraciu\\ Department of Mathematics\\ University of South Carolina\\
Columbia\\ SC 29208\\ U.S.A.} \email{vraciu@math.sc.edu}


\thanks{Research partly supported by  NSA grant H98230-10-1-0361 and Simons Foundation grant 233597 (ARK), NSF grant DMS-1101131 and Simons Foundation grant 20903 (LMS), NSA grant  H98230-11-1-0176 and NSF grant  DMS-1200085 (AV)}

\begin{abstract}
A quasi-complete intersection (q.c.i.)  ideal of a local ring is an ideal with ``free exterior Koszul homology"; the definition can also be understood in terms of vanishing of  Andr\'e-Quillen homology functors.  Principal \qci ideals are well understood, but few constructions are known to produce \qci ideals of grade zero that are not principal. This paper examines the structure of q.c.i.\ ideals. We exhibit conditions on a ring $R$ which guarantee that every q.c.i.\ ideal of $R$ is principal. On the other hand, we give an example of a minimal q.c.i.\ ideal $I$ which does not contain any principal q.c.i.\ ideals and is not embedded, in the sense that no faithfully flat extension of $I$  can be written as a quotient of complete intersection ideals.  We also describe a generic situation in which the maximal ideal of $R$ is an embedded q.c.i.\ ideal that does not contain any principal q.c.i.\ ideals.
 
\end{abstract}
\maketitle

\section*{introduction}

This paper is concerned with a class of ideals referred to as  {\it quasi-complete intersection {\rm (q.c.i.)} ideals} in recent work of Avramov et.\,al.\,\cite{AHS}.  As discussed there, the notion goes back to work of Rodicio \cite{RoCMH}, and appears in subsequent papers of A.~Blanco, J.\,Majadas Soto and A.\,Rodicio Garcia.  Q.c.i.\ ideals of local rings can be defined as ideals with ``free exterior Koszul homology", see Definition \ref{qci-def}. The class of q.c.i.\ ideals contains that of {\it complete intersection ideals} (i.e.\,ideals generated by a regular sequence), and inherits many of the homological change of rings properties of the latter. 

Consult \cite{BMR,AHS} for the connection between q.c.i.\ ideals and the vanishing of Andr\'e-Quillen homology. In particular, an ideal $I$ is a q.c.i.\ if and only if the homomorphism $R\to R/I$ satisfies the conclusion of the Quillen conjecture \cite[5.6]{Qu}. 
 Not many examples or methods of constructing such homomorphisms are known, and a better understanding of 
\qci ideals 
is of value as one tries to prove or disprove the conjecture.

Our goal is to understand the structure of q.c.i.\ ideals $I$ of a  commutative local  noetherian ring $(R,\fm)$; this notation identifies $\fm$ as the maximal ideal of the local ring $R$.  Principal q.c.i.\ ideals  are well understood:  If $x\ne 0$ is an element of $\fm$, then the ideal $(x)$ is q.c.i.\ if and only if $x$ is either regular or else $\ann(x)\cong R/(x)$; in the last case we say, following Henriques and \c Sega \cite{HS}, that $x$ is an  {\it exact zero-divisor}. (Such elements are studied also in \cite{KST} under a slightly different name.)  Existence of exact zero-divisors is known for certain classes of small artinian rings, and has found various  uses, see \cite{CRV}, \cite{AIS} and \cite{HS}.  Another well-understood method of constructing q.c.i.\ ideals is by means of a pair of embedded complete intersection ideals; see Remark \ref{embedded}.  The q.c.i.\ ideals $I$ that can be obtained in this manner after possibly a faithfully flat extension  are exactly the ones for which $\cidim_R(R/I)<\infty$, where $\cidim$ denotes complete intersection dimension, as defined in \cite{AGP}.
We say that such q.c.i.\ ideals are {\it embedded.}

New  q.c.i.\ ideals can be constructed from old ones by  ``composition" and ``decomposition" of surjective \qci homomorphisms, see  \cite[8.8,8.9]{AHS}. In particular, if $I=(a_1,\dots,a_s)$ is an ideal in $R$ with $a_{i+1}$ an exact zero-divisor or a regular element on $R/(a_1,\dots,a_i)$ for all $i$ with $0\le i\le s-1$, then $I$ is a q.c.i.\ ideal of $R$. Following the lead of \cite{KST} and \cite[\S 3]{AHS}, we call such ideals {\it exact}. We say that a \qci ideal $J$ is {\it minimal} if $J$ does not properly contain any non-zero q.c.i.\ ideal.

Rodicio \cite[Conjecture 11]{SoJPAA} conjectured that all q.c.i.\ ideals of $R$ are embedded. Although this statement holds  under some additional  conditions on the ring $R$, see \cite[Proposition 23]{SoJPAA},  in general it does not. A counterexample consisting of a  principal  non-embedded q.c.i.\ ideal  is given in \cite[Theorem 3.5]{AHS}; one can further argue that this \qci ideal is minimal, see Proposition  \ref{3.5}.  

 Beyond the information mentioned  so far, the literature seems to  lack  other relevant examples and methods of constructing  q.c.i.\ ideals. In this paper  we further clarify the structure of \qci ideals
and in particular examine relations between the classes of q.c.i.\ ideals (principal, exact, minimal, embedded) introduced above.

Note that a minimal  q.c.i.\ ideal is exact if and only if it is principal. In Section \ref{Loewy} we  discuss some classes of rings for which every q.c.i.\ ideal is principal (thus  exact) as follows: 

\begin{Theorem1}
Let $(R,\fm)$ be an artinian local ring which is not a complete intersection.  Assume that one of the following holds: 
\begin{enumerate}[\quad\rm(1)]
\item $\fm^3=0$; 
\item $\fm^4=0$ and $R$ is Gorenstein. 
\end{enumerate}
Then every q.c.i.\ ideal of $R$ is principal. 
\end{Theorem1}

Theorem 1 is part of Theorem \ref{loewythm}, which studies, more generally,  bounds on the minimal number of generators of a q.c.i.\ ideal.


On the other hand, Proposition \ref{details} and Theorem \ref{summary-4} give: 

\begin{Theorem2}
There exists an artinian local ring $(R,\fm)$ with $\fm^4=0$ and elements $f_1, f_2\in \fm$ that are linearly independent modulo $\fm^2$ and generate a minimal,  non-embedded and non-principal (thus non-exact)  q.c.i.\ ideal. 
\end{Theorem2}


While the example involved in the proof of this result  is rather special, in Section 5 we 
exhibit many grade zero embedded q.c.i.\ ideals which are not exact. 

\begin{Theorem3}
Let $k$ be an algebraically closed field of characteristic different from $2$ and let $P$ denote the polynomial ring $k[x_1,\dots, x_n]$. 

If $n\ge 5$ and ${\bd f}=f_1, \dots, f_n$ is a generic regular sequence  of quadratic forms, then $(x_1, \dots, x_n)/{\bd f} P$ is an embedded \qci ideal of $R=P/{\bd f}P$ that does not contain any principal $\qci$ ideal. 
\end{Theorem3}
 
Theorem 3 is  part of Theorem \ref{5.2}. The meaning of the word ``generic" is made precise through Theorem \ref{5.1}. 

Preliminaries and general results on \qci ideals are collected in Sections 1 and 2. In particular,  Corollary \ref{cor exact pair} gives a necessary condition for the existence of exact zero-divisors:  If $R=Q/\fa$ with $(Q,\fn,k)$ a regular local ring and $\fa\subseteq \fn^2$, and $R$ admits an exact zero-divisor, then a minimal generator of $\fa$ factors non-trivially.

\section{Preliminaries}

In this section we present and discuss the notion of \qci ideal.  A criterion for checking that a 2-generated  ideal  of grade zero is \qci is given in Lemma \ref{2-gen-form}.  Theorem \ref{mingens} and Propositions \ref{exact pair} and \ref{homogeneous} give some consequences of the \qci property.
 
The following notation and conventions are used throughout the paper: 
Let $(R,\fm,k)$ be a local ring: $R$ is a commutative noetherian ring 
with unique maximal ideal $\fm$, and $k=R/\fm$. If $M$ is a finitely generated $R$-module, we denote by $\nu(M)$ the minimal number of generators of $M$. 

Let $I$ be an ideal 
of $R$ with $\nu(I)=n$ and set $S=R/I$.  Let $\bd f=f_1, \dots, f_n$ be a generating set of $I$ and let $E$ denote the Koszul complex 
on $\bd f$. 

\begin{definition}
\label{qci-def} We say that $I$ is a \emph{quasi-complete intersection  $($q.c.i.$)$  ideal} if $\HH_1(E)$ is free over $S$  and the canonical homomorphism of graded $S$-algebras
 \begin{equation}
    \label{eq:lambda}
\lambda^S_*\colon \bw_*^S \HH_1(E)\lra\HH_{*}(E)
  \end{equation}
is bijective, where $\bw_*^S$ denotes the exterior algebra functor.
\end{definition}

 We refer to \cite{BMR} for the interpretation of the notion of q.c.i.\ in terms of vanishing of Andr\'e-Quillen homology functors. 

\begin{bfchunk}{Principal q.c.i.\ ideals.}
\label{principal}
We say that an element $x$ of $R$ is an {\it exact zero-divisor} if 
$$
 R\ne (0:_Rx)\cong R/(x)\ne 0\,.
$$
If $x$ is an exact zero-divisor, then there exists $y$ such that $(0:_Rx)=(y)$ and $(0:_Ry)=(x)$. We say that $x,y$ is an {\it exact pair of zero-divisors} and $y$ is a {\it complementary zero-divisor of} $x$. Such pairs were first  studied  in \cite{KST} under the name of {\it exact pairs of elements}; the name exact zero-divisor was introduced in \cite{HS}. 

It follows directly from Definition \ref{qci-def} that a non-trivial principal ideal $I=(x)$ is q.c.i.\ if and only if $x$ is either a non zero-divisor or an exact zero-divisor. 
\end{bfchunk}

\begin{chunk}
Recall that $\grade_R(I)$ denotes the maximal length of an $R$-regular sequence in $I$; this number is equal to the least integer $i$ with $\Ext^i_R(R/I,R)\ne 0$. 
In view of \cite[Lemma 1.4]{AHS}, the study of the structure of q.c.i.\ ideals may be reduced to the case when  $\grade_R(I)=0$. 

If $I$ is a \qci ideal, then \cite[1.2]{AHS} gives:
\begin{equation}
\label{grade}
\grade_R(I)=\nu(I)-\nu(\HH_1(E)).
\end{equation}
\end{chunk}

\begin{chunk}\label{9-18}
\label{lambda}  Assume that  $\nu(\HH_1(E))=n$. This assumption holds whenever $I$ is a \qci ideal with $\grade_R(I)=0$ by \eqref{grade}; however, we do not want to assume that $I$ is \qci at this time. 

Since $\nu(I)=n$, we have $E_1\cong R^n$. Let $v_1,\dots, v_n$ denote a basis of $E_1$ with $\dd(v_i)=f_i$ for each $i$. Consider a set of cycles 
\begin{equation}
\label{z}
z_j=\sum_{i=1}^n a_{ij}v_i
\end{equation}
 with $a_{ij}\in R$ and $j=1,\dots ,n$ such that the homology classes $\cls(z_j)$ minimally generate $\HH_1(E)$. 
 Set 
$$
A=(a_{ij})\quad\text{and}\quad \Delta =\det(A)
$$
and note that the map 
$$
\lambda:=\lambda^S_n\colon \bw_n^S \HH_1(E)\lra\HH_{n}(E)
$$
is described by 
$$\lambda\left(\cls(z_1)\wedge\dots\wedge \cls(z_n)\right)=\Delta  v_1\dots v_n\,.$$
Note that $\Delta \in (0:_RI)$. Since $\bd f$ is a minimal generating set for $I$, and each $z_i$ is a syzygy in the free cover $E_1\to I$, we have $a_{ij}\in \fm$ for all $i,j$. In particular, we have: 
\begin{equation}
\label{alpha}
\Delta \in \fm^n.
\end{equation}
\end{chunk}

\begin{lemma}
\label{det}
If $I$ is a q.c.i.\ ideal with $\grade_R(I)=0$, then the following hold: 
\begin{enumerate}[\quad\rm(1)]
\item $\HH_1(E)\cong S^n$;
\item $(0:_RI)\cong S$;
\item $(0:_RI)=\Delta  R$ and $(0:_R\Delta )=I$.
\end{enumerate} 
\end{lemma}

\begin{proof}
Since $\grade_R(I)=0$, we know that $\nu(\HH_1(E))=n$ by \eqref{grade}. Then (1) follows from the fact that $\HH_1(E)$ is free over $S$, according to Definition \ref{qci-def}. 

(2) We have 
$$S\cong \bw_n^S (S^n)\cong \bw_n^S \HH_1(E)\cong \HH_n(E)\cong(0:_RI)\,,
$$
where the third isomorphism is given by the map $\lambda$ in \ref{lambda}. The first and the last isomorphism are general facts, and the second one is a consequence of (1). 

(3) Using the description of the map $\lambda$ in \ref{lambda}, we see that the isomorphism $S\xrightarrow{\cong} (0:_RI)$ from the proof of (2) can be described by 
$$
1\mapsto \Delta 
$$
In particular, $(0:_RI)=\Delta  R$. The fact that this map is an isomorphism shows that $(0:_R \Delta )=I$.
\end{proof}

\begin{remark}\label{Oct13} For any q.c.i.\ ideal $I$ the module $R/I$ has a Tate resolution $T$ with
$T_1 = R^n$, $T_0 = R$ and $d_1 = [f_1 \cdots f_n]$ (see \cite[1.5 and 1.6]{AHS}). When $\operatorname{grade}_R(I) = 0$
it gives rise to an infinite in both directions exact sequence
$$\cdots \longrightarrow T_1\stackrel{d_1}{\longrightarrow}T_0\stackrel{d_0}{\longrightarrow}T_1^*\longrightarrow \cdots$$
where $d_0$ is given by multiplication with $\Delta$. Indeed, $\operatorname{H}^n(T^*) = 0 $ for $n \ge 1$ and
$\operatorname{H}^0(T^*) \simeq S$ by \cite[Thm. 2.5(4)]{AHS}; Lemma~\ref{det} gives exactness at $T_0$ and $T_0^*$.
\end{remark}

As noted in \ref{principal}, principal q.c.i.\ ideals admit a simple characterization. Based on  Lemma \ref{det} and Definition \ref{qci-def}, the two-generated q.c.i.\ ideals can also be given a relatively simple characterization as follows.

\begin{lemma}\label{2-gen-form}  
Let $I$ be an ideal with $\nu(I)=2$ and $\grade_R(I)=0$. Then the following statements are equivalent: 
\begin{enumerate}[\quad\rm(1)]
\item $I$ is q.c.i.\  
\item  $\HH_1(E)\cong S^2$, $(0:_RI)=\Delta  R$, and $(0:_R\Delta )=I$, where $\Delta$ is defined as in  {\rm \ref{lambda}}.
\item There exist elements $a,b,c,d$ in $\mathfrak m$ with {\rm(\ref{abcd})} an exact sequence, where  
\begin{equation}\label{abcd}
 \xymatrix{R^4\ar[r]^{d_3}&R^3\ar[r]^{d_2}&R^2\ar[r]^{d_1}&R\ar[r]^{d_0}&R\ar[r]^{d_1^{\rm T}}&R^2,}\end{equation}
  with $d_0=[ad-bc]$, $d_1=\bmatrix f_1&f_2\endbmatrix $, 
$$    d_2=\bmatrix -f_2&a&b\\f_1&c&d\endbmatrix,\quad\text{and}\quad d_3=\bmatrix -c&-d&a&b\\f_1&0&f_2&0\\ 0&f_1&0&f_2\endbmatrix.$$   
\end{enumerate}
\end{lemma} 

\begin{proof}

 $(1)\Rightarrow (3)$: By Lemma \ref{det} there are cycles $z_1 = av_1 + cv_2$ and $z_2 = bv_1 + dv_2$
in $E_1$, whose classes form a basis of $\operatorname{H}_1(E)$. Let $T$ be the Tate resolution of $R/I$
constructed with these cycles. Then Remark \ref{Oct13} gives the exact sequence in (\ref{abcd}).

\smallskip\noindent $(3)\Rightarrow(2)$:  Most of the hypotheses of (2) follow immediately from (3). We only need to verify  that $\HH_1(E)\cong S^2$.  The exactness of the complex \eqref{abcd} implies that $\nu(\HH_1(E))=2$. Furthermore, the cycles $z_1$ and $z_2$ in \ref{lambda} can be taken to be
$$
z_1=av_1+cv_2\quad \text{and}\quad z_2=bv_1+dv_2.
$$
Consider the  homomorphism 
$$\varphi\colon R^2\twoheadrightarrow \operatorname{H}_1(E)$$  given by $\varphi(t)=\cls(t_1 z_1+t_2 z_2)$, where $t=\bmatrix t_1,t_2\endbmatrix^{\text{\rm T}}\in R^2$.  
If $\varphi(t)=0$ , then there exists $t_0\in R$ such that the element $\bmatrix t_0,t_1,t_2\endbmatrix^{\text{\rm T}}\in R^3$  is  in $\Ker d_2=\Ima d_3$. By looking at the matrix describing $d_3$, we conclude that $t_1$ and $t_2$ are in $I$. It follows that $\Ker(\varphi)\subseteq IR^2$. The reverse inclusion is clear, hence  $\HH_1(E)\cong S^2$.

\smallskip\noindent $(2)\Rightarrow(1)$:
Let $v_1,v_2$ be an $R$-module basis for $E_1$.  
It follows that $$\textstyle \operatorname{H}_2(E)=\{rv_1v_2\in \bigwedge^2 E_1\mid r\in (0:_RI)\}.$$
The hypothesis $\operatorname{H}_1(E)\cong S^2$ of (2) guarantees that there exist cycles $z_1$ and $z_2$   in $E_1$ such that $\operatorname{cls} (z_1)$ and $\operatorname{cls} (z_2)$ form a  basis for the free $S$-module $\operatorname{H}_1(E)$. We know from \ref{9-18} that the $S$-module homomorphism 
$\lambda_2^S:\bigwedge^2(\operatorname{H}_1(E))\to \operatorname{H}_2(E)$ is given by
$\lambda (r\operatorname{cls}(z_1)\wedge \operatorname{cls}(z_2))=r\Delta(v_1v_2)$.
The hypotheses    $(0:_RI)=\Delta  R$, and $(0:_R\Delta )=I$ of (2) ensure that $\lambda_2$ is an isomorphism of $S$-modules.
  \end{proof}

\begin{theorem}
\label{mingens} Let $(Q,\fn,k)$ be a regular local ring,  $\fa\subseteq \fn^2$ be an ideal of $Q$,  $(R,\fm,k)$ be the local ring with $R=Q/\fa$ and  $\fm=\fn/\fa$, and $J$ be an ideal of $Q$ which contains $\fa$.
If $I=J/\fa$ is a q.c.i. ideal of $R$, then $$\nu(J)=\nu(\fa)+\grade_R(I)\,.$$
\end{theorem} 

\begin{proof}
Set $H=\HH_1(E)$ and $S=Q/J=R/I$.
By \cite[Theorem 5.3]{AHS}  or \cite{R} (in view of \cite[Remark 5.3]{AHS}),  we have an exact sequence: 
 \begin{equation}
    \label{eq:sixterms}
0\to H/\fm H\to\pi_2(R)\xrightarrow{\pi_2(\varphi)}\pi_2(S)
\xrightarrow{\,\delta \,} I/\fm I\to\pi_1(R)
\xrightarrow{\pi_1(\varphi)}\pi_1(S)\to 0
  \end{equation}
We refer to \cite[5.1,\,5.2]{AHS} for the definition of  the modules of indecomposables $\pi_i(-)$. 
According to \cite[5.2]{AHS} and the proof of \cite[Prop.~3.3.4]{GL}, we have canonical identifications $\pi_1(R)=\fn/\fn^2$ and $\pi_2(R)=\fa/\fn \fa$. In particular:
$$
\rank_k\pi_1(R)=\nu(\fn)\quad\text{and}\quad\rank_k\pi_2(R)=\nu(\fa)\,.
$$
Choose $b_1, \dots, b_s$ in $J$ so that their images form a basis of the kernel of the induced map $\fn/\fn^2\to \fn S/(\fn S)^2$. The local ring $\ov Q=Q/(b_1, \dots, b_s)$ is regular with maximal ideal $\ov{\fn}=\fn/(b_1, \dots, b_s)$ and $S\cong \ov Q/\ov J$, where $\ov J=J/(b_1,\dots, b_s)\subseteq {\ov \fn}^2$. 
We then have
\begin{align*}
\rank_k \pi_1(S)&=\nu(\ov{\fn})=\nu(\fn)-s=\nu(\fm)-s\\
\rank_k \pi_2(S)&=\nu(\ov J)=\nu(J)-s
\end{align*}
The Euler characteristic of \eqref{eq:sixterms} computed 
using the expressions above and 
the relation  $\nu(H)=\nu(I)-\grade_R(I)$  from \eqref{grade}  gives 
$\nu(J)=\nu(\fa)+\grade_R(I)$. 
\end{proof}

\begin{proposition}\label{exact pair}
Let $(Q,\fn,k)$ be a regular local ring. Let $\fa\subseteq \fn^2$ be an ideal and set $R=Q/\fa$. Let $F,G\in Q$ such that $FG\in\fa$, and let $f$, $g$ denote the images of these elements in $R$. 

  If $f, g$ is an exact  pair of zero-divisors, then $FG\notin \fn\fa$. Furthermore, if $\fa$ is generated by a regular sequence, then the converse holds.
\end{proposition}

\begin{proof}
Assume $f, g$ is an exact  pair of zero-divisors. 
We apply Theorem \ref{mingens} to the \qci ideal $I=J/\fa$ with $J=\fa +(F)$. The proposition yields that the ideals $\fa$ and $\fa+(F)$ have the same minimal number of generators. Fix a minimal generating set of $\fa$.  Then one of the minimal generators of $\fa$ can be written as a linear  combination of the remaining  minimal generators of $\fa$ and $F$, and hence there exists an element $H\in Q$ such that $FH\in \fa\smallsetminus \fn\fa$. Since $FH\in \fa$, we see that $h\in (0\colon f)=(g)$, where $h$ denotes the image of $H$ in $R$. Hence $H=YG+X$ with $X\in \fa$ and $Y\in Q$. Thus $FH-YFG\in \fn \fa$, and we conclude $FG\notin \fn\fa$, since $FH\notin\fn\fa$. 

Assume now that $\fa$ is generated by a regular sequence. 
Assume that $FG\notin \fn\fa$. In particular, $FG$ is minimal generator of $\fa$ and can be completed to a minimal generating set for $\fa$, say $a_1, a_2, \dots, a_r,FG$. Since $\fa$ can be generated by a regular sequence, its minimal generating set $a_1, a_2, \dots, a_r, FG$ is itself a regular sequence. It follows that  $a_1, a_2, \dots, a_r, F$  is a regular sequence as well. Using this information, one can easily argue that $(0:_Rf)=(g)$, and similarly $(0:_Rg)=(f)$. 
\end{proof}

Let $k$ be a field. We let $P=k[x_1, \dots, x_e]$ denote the polynomial ring in $n$ variables of degree $1$, and we set $\fp=(x_1,\dots, x_e)P$. We take $Q=k[[x_1,\dots, x_e]]$ to be the power series ring, with maximal ideal  $\fn=(x_1,\dots, x_e)Q$. If $h\in Q$, we denote by $h^*$ the initial form of $h$ (which can be regarded as both an element of $P$ and of $Q$).

\begin{proposition}
\label{homogeneous} 

Let $\fb$ be a homogeneous ideal of $P$ and set $\fa=\fb Q$, where $Q=k[[x_1, \dots, x_e]]$.    If $y\in\fa$, then $y^*\in \fb$. Furthermore, if  $\fb$ is generated by homogeneous polynomials of the same degree, then the following hold: 
 \begin{enumerate}[\quad\rm(a)]
 \item If $y\in \fa$, then $y-y^*\in \fn\fa$.  
 
\item If  $F$, $G$ are elements of $Q$ such that their images $f$, $g$ in $Q/\fa$ form an exact pair of zero-divisors,  then $F^*G^*\notin\fn\fa$.
\end{enumerate}
\end{proposition}

\begin{proof}
For each integer $i$ one has canonical  isomorphisms 
$$
P/\fp^i\cong Q/\fn^i
$$
which allow one to translate the statements to  a graded setting, where they are clear.  If $y\in\fa$, it follows that  $y^*\in \fb+\fp^i$  for each $i$, hence $y^*\in \fb$. 
 
 Assume now that $\fb$ is generated by homogeneous polynomials of the {\it same} degree. 

(a) If $y\in \fa$, then $y-y^*\in \fn\fa+\fn^i$ for all $i>\deg(y^*)$, hence $y-y^*\in \fn\fa$. 

(b) Assume that $F, G$ are such that $f, g$ form a pair of exact zero-divisors. By Proposition \ref{exact pair} we know that $FG\notin \fn\fa$. Part (a) gives then that $FG-(FG)^*\in \fn\fa$, hence $F^*G^*=(FG)^*\notin \fn\fa$. 
\end{proof}

\begin{corollary}\label{cor exact pair}
Let $(Q, \mathfrak n, k)$ be a regular local ring and $\mathfrak a \subseteq \mathfrak n^2$. If $R = Q/\mathfrak a$ contains an exact
zero-divisor, then $\mathfrak a$ has a minimal generator $fg$ with $f, g \in \mathfrak n$. Furthermore, if $Q$
is a power series ring over $k$ and $\mathfrak a$ is generated by homogeneous polynomials of the same degree, then
$f$ and $g$ can be chosen to be homogeneous polynomials.\qed
\end{corollary}

\section {Embedded \qci ideals} 

In this section we define the notion of embedded \qci ideal. We spell out  a known characterization of such ideals in Remark \ref{embedded}. We are mainly interested in finding a procedure for checking 
that a given \qci ideal is {\bf not} embedded. This is achieved in Lemma \ref{non-emb}, by using the terminology of homotopy Lie algebra. The approach used here expands the one in the proof of \cite[Theorem 3.5]{AHS}. 

 \begin{chunk}
    \label{ch:cidim}
A \emph{quasi-deformation} is a pair $R\to R'\gets Q$ of homomorphisms
of local rings, with $R\to R'$ faithfully flat and $R'\gets Q$ surjective with 
kernel generated by a $Q$-regular sequence.  By definition, the \emph{CI-dimension}
of an $R$-module $M$, denoted $\cidim_RM$, is finite if $\pd_Q(R'\otimes_RM)$ 
is finite for some quasi-deformation; see~\cite{AGP}.
 \end{chunk}

If $M$ is a finitely generated  $R$-module, then its $n$th {\it betti number} is the integer 
$$\beta_n^R(M)=\rank_k(\Tor_n^R(M,k))\,.$$ 

\begin{chunk}
Consider the following conditions concerning an ideal $I$ of the local ring $R$: 
\begin{enumerate}[\quad\rm(1)]
\item $\cidim_R(R/I)<\infty$ and $\HH_1(E)$ is a free $R/I$-module. 
\item $I$ is a q.c.i.\ ideal. 
\item The betti numbers of the $R$-module $R/I$ have polynomial growth and $\HH_1(E)$ is a free $R/I$-module. 
\end{enumerate}
Soto \cite[Proposition 23]{{SoJPAA}} shows that the implications (1)$\implies$(2)$\implies$(3) always hold, and that the three statements are equivalent for certain classes of rings, for which the asymptotic behavior of betti numbers is well understood. 

Rodicio also  conjectured that (1)$\iff$(2) always holds. As discussed in the Introduction, \cite[Theorem 3.5]{AHS} provides a  counterexample with $I$ a principal ideal. 
\end{chunk}

In what follows, we say that an ideal of a ring $Q$ is a {\it complete intersection} ideal if it can be generated by a $Q$-regular sequence.

\begin{remark}
\label{embedded}
The following statements are equivalent: 
\begin{enumerate}[\quad\rm (1)]
\item $I$ is a q.c.i. ideal and $\cidim_R(R/I)<\infty$. 
\item There exists a faithfully flat extension $R\to R'$, a local ring $Q$ and complete intersection ideals $\fa\subseteq \fb$ of $Q$  such that $R'=Q/\fa$ and $R'/IR'=Q/\fb$.
\end{enumerate}

The implication (1)$\implies$(2) is given by \cite[2.7]{AHS} and the converse follows from \cite[1.3,\,1.4]{AHS}. 
\end{remark}

To simplify the terminology and better convey the structural property described in Remark \ref{embedded}(2), we introduce the following definition:

\begin{definition} We say that a q.c.i.\ ideal $I$ of $R$ is {\it embedded} if $\cidim_R(R/I) < \infty$.
\end{definition}

\begin{bfchunk}{Complexity.}
If $M$ is a finitely generated $R$-module, the {\it complexity of $M$}, denoted $\cx_R(M)$, is the least integer $d$ such that there exists a polynomial $f(t)$ of degree $d-1$ such that $\beta_i^R(M)\le f(i)$ for all $i\ge 1$.


If $I$ is a q.c.i.\ ideal of $R$, then   \eqref{grade} and the minimality of Tate's resolution (see \cite[1.5, 1.6]{AHS}) yield
\begin{equation}
\label{lem}  \nu(I) -\operatorname{grade}_R(I) = \operatorname{rank}_R \operatorname{H}_1(E) = \operatorname{cx}_R(R/I).\end{equation}

\end{bfchunk}

\bigskip
Next, we extend an argument used in the proof of \cite[3.5]{AHS}. 

\begin{bfchunk}{The homotopy Lie algebra.}
\label{Lie}
It is known that there exists a graded Lie algebra over $k$, denoted $\pi^*(R)$ such that the universal enveloping algebra of $\pi^*(R)$ is equal to the algebra $\Ext^*_R(k,k)$ with Yoneda products, see \cite[\S 10]{Avr98} for details. We let $\zeta^*(R)$ denote the center of $\pi^*(R)$.
\end{bfchunk}

\begin{lemma}\label{non-emb}  If I is an embedded q.c.i. ideal, then $\nu(I)-\operatorname{grade}_R(I) \le \operatorname{rank}_k \zeta^2(R)$.\end{lemma}

\begin{proof} By \cite[5.3]{AS}, $\operatorname{Ext}_R(R/I, k)$ is a finitely generated module over the symmetric
algebra $\mathcal P$ of $\zeta^2(R)$, and its Krull dimension equals $\operatorname{cx}_R(R/I)$ by \cite[5.3]{AGP}. Now
(\ref{lem}) and elementary properties of Krull dimension give
$$\nu(I) -\operatorname {grade}_R(I) = \operatorname{cx}_R(R/I) = \dim_{\mathcal P} \operatorname{Ext}_R(R/I, k) \le \dim \mathcal P = \operatorname{rank}_k \zeta^2(R). $$ \end{proof}

\section{Loewy length and minimal generation of \qci ideals}\label{Loewy}

If the local ring $(R,\fm,k)$ is artinian, then its {\it Loewy length} is defined as the number 
$$\ell\ell(R)=\inf\{l\ge 0\mid \fm^l=0\}\,.
$$
 In this section we show that the number of generators of a \qci ideal of $R$ can be bounded in terms of $\ell\ell(R)$. 

We say that $R$ is a {\it complete intersection} ring   if $\widehat R=Q/\fa$ for a regular local ring $Q$ and a complete intersection ideal $\fa$.

\begin{chunk}
\label{ci and Gor}
If $I$ is a \qci ideal of $R$, then the following statements are equivalent (see for example  \cite[Prop.~7.7]{AHS} and \cite[Cor.~7.6]{AHS}):
\begin{enumerate}[\quad\rm(1)]
\item $R$ is Gorenstein, respectively complete intersection;
\item $R/I$ is Gorenstein, respectively complete intersection. 
\end{enumerate}
\end{chunk}

The main result of this section is as follows. Note that properties (2) and (5) below yield immediately the statement of Theorem 1 in the Introduction.  

\begin{theorem}
\label{loewythm}
Let $(R, \fm,k)$ be a local  artinian ring. Let $I \subset R$ be a nontrivial  q.c.i. ideal and set $l=\ell\ell(R)$. The following then hold: 
\begin{enumerate}[\quad\rm(1)]
\item $\nu(I)\le l-1$;
\item If $R/I$ is not a complete intersection, then $\nu(I)\le l-2$;
\item If $\nu(I)=l-2$ and $I\cap \fm^2\subseteq \fm I$,  then 
$\nu(\fm/I)\le \nu(\fm^{l-1})\,;$
\item If $R/I$ is Gorenstein, not  a complete intersection, and $I\cap \fm^2\subseteq \fm I$, then $\nu(I)\le l-3$;
\item If $R/I$ is Gorenstein, not a complete intersection, then $l\ge 4$.  If $l=4$, then $\nu(I)=l-3=1$. 
\end{enumerate}
\end{theorem}

\begin{remark}
We can argue that the bounds in the theorem are sharp, by pointing out extremal examples. 

For (1), consider the ring $R=k[[X_1, \dots, X_n]]/(X_1^2,\dots, X_n^2)$.
The ideal $I = (x_1,\ldots,x_n)$ is an embedded q.c.i.\ with $\nu(I) = n$ and $l = n + 1$.


For (2) and (3), consider for example the ring $R$ and the ideal $I$ in Section 4, for which  $\nu(I)=2$, $l=4$, $\nu(\fm)=5$ and $\nu(\fm^3)=3$. 

For (4), any generic Gorenstein algebra with $\fm^4=0$ and $\nu(\fm)\ge 3$ works,  since such a ring is known (see \cite[Rmk.~4.3]{HS})  to have an exact zero-divisor, so that one can take $I$ with $\nu(I)=1$. 
\end{remark}

\begin{proof}
Set $n=\nu(I)$. We use the notation in \ref{lambda}.  By Lemma \ref{det} we have $(0:_RI)=\Delta  R$ and  $(0:_R\Delta)=I$. Also, \eqref{alpha} gives $\Delta \in \fm^n$. 

(1) Note that $\Delta \ne 0$, hence $\fm^n\ne 0$. 

(2) Assume that $n=l-1$. Since $\Delta \in \fm^{l-1}$, we have $\Delta  \in (0:_R\fm)$. On the other hand, we have $(0:_R\fm)I=0$, hence $(0:_R\fm)\subseteq (0:_RI)=\Delta  R$. It follows that $(0:_R\fm)=\Delta  R$. In particular, $R$ is Gorenstein. We conclude $I=(0:_R\Delta )=\fm$. Hence $I=\fm$ is a \qci ideal. Using \ref{ci and Gor} we conclude that $R$ is a complete intersection, a contradiction. 

 (3) Since $I\cap \mathfrak m^2 \subseteq \mathfrak mI$, the ideal $\mathfrak m$ has a minimal generating set $f_1,\ldots,f_n, h_1, \allowbreak \ldots, \allowbreak h_t$ such that $f_1,\ldots,f_n$ minimally generate $I$. Assuming $n=l-2$, we have $\Delta \in \fm^{l-2}$, and thus 
$h_i \Delta  \in \fm^{l-1}$ for all $i$.   Note that the elements $h_i\Delta $ of $\fm^{l-1}$ are linearly independent. Indeed, if $\sum c_ih_i\Delta =0$ for some constants $c _i$, not all zero,  then it would follow $\sum_i c_ih_i\in(0:_R\Delta )=I=(f_1, \dots,f_n)$, a contradiction. It follows that $t\le \rank_k(\fm^{l-1})=\nu(\fm^{l-1})$. 

(4) By (2), we know that $n\le l-2$. Assume $n=l-2$. Then (3) gives that $\nu(\fm/I)\le 1$. Note that $R/I$ is Gorenstein by  \ref{ci and Gor}.  The ring $R/I$ is thus a Gorenstein ring of embedding dimension $1$; it is thus a complete intersection, and hence $R$ is a complete intersection by \ref{ci and Gor}.

(5) By (2), we have $l\ge 3$. Assume $l=3$ and $\nu(I)=1$. If $I=(f)$, then (4) shows that  $f\in\fm^2$. Since $f\fm=0$, it follows that  $\fm\subseteq (0:f)=(\Delta )$. Thus $\fm$ is $1$-generated, and it follows that $R$ is a complete intersection, a contradiction. 

Assume now that $l=4$.  If $I$ is not principal, then it can be minimally generated by two elements. Let $I=(f_1,f_2)$. By (3), we may assume that one of these elements is in $\fm^2$. Assume $f_1 \in \fm ^2$ and note that $f_2 \notin \fm ^3$.

Let $\fm^3 = (\delta )$ be the socle of $R$. For every $x\in \fm$ we have $xf_1 \in \fm ^3$, and therefore $xf_1 =\alpha_x \delta $ where $\alpha _x$ is either zero or a unit in $R$. If $\alpha _x =0$ then we  take $y_x =0$; if $\alpha_x$ is a unit we use the fact that there exists a non-zero multiple of $f_2$  in the socle to find  $y_x$ such that $y_x f_2 = xf_1$. Since $f_2 \notin \fm ^3$, we have $y_x \in \fm $.

In either case there exists $y_x \in \fm$ such that $xf_1 =y_x f_2$. With the notation in \ref{lambda}, the elements 
$$
xv_1-y_xv_2
$$
are cycles in the Koszul complex $E$. Since $\nu(I)=2$, we have that $\nu(\HH_1(E))=2$. Let $z_1$ and $z_2$ be the two cycles  in \ref{lambda} whose classes generate $\HH_1(E)$, with 
$$
z_j=a_{1j}v_1+a_{2j}v_2
$$
It follows that for every $x\in \fm$, the element $xv_1-y_xv_2$ is a linear combination of $z_1$, $z_2$ and the boundary $f_2v_1-f_1v_2$. Consequently, $\fm=(a_{11},a_{12},f_2)$. The ring $R/I$ is then Gorenstein and has embedding dimension at most $2$. It is thus a complete intersection, and thus $R$ is a complete intersection, a contradiction. 
\end{proof}

\begin{definition}
We say that a \qci ideal $I$ is {\it minimal } if $I$  does not properly contain any non-zero q.c.i.\ ideal.
\end{definition}

\begin{remark}\label{3.6}
If $I$ is a minimal \qci ideal, then $\grade_R(I)=0$, because
every regular element generates a q.c.i.\ ideal. 
\end{remark}

The results proved so far allow us to show that certain \qci ideals are minimal.

\begin{proposition}\label{3.5}
Let $R=Q/\fa$ be an artinian local ring, where $(Q,\fn,k)$ is a regular local ring and $\fa\subseteq \fn^2$. If $R$ is not a complete intersection, $\ell\ell(R)=3$ and $\fa\cap \fn^3\subseteq \fa\fn$, then any \qci ideal of $R$ is minimal. 

In particular, the ideal $I$ of {\rm \cite[Theorem 3.5]{AHS}} is a minimal \qci ideal. 
\end{proposition}

\begin{proof}
By Theorem \ref{loewythm}(2), any \qci ideal of $R$ is principal. Let $I=(h)$ with $h\in \fm$ be a \qci ideal.  If $J\subseteq I$ is another \qci ideal with $J\ne I$ then $J=(f)$ and $f=ah$ with $a\in \fm$. In particular, $f\in \fm^2$. If $g$ is a complementary zero-divisor of $f$, and $F$ and $G$ are the liftings of these elements in $Q$, \ref{exact pair} shows that $FG$ is a minimal generator of $\fa$.  Since $FG\in\fn^3$, this contradicts the hypothesis that $\fa\cap \fn^3\subseteq \fa\fn$. 
\end{proof}

\section{A non-principal, non-embedded, minimal q.c.i.\ ideal}\label{example}


In this section we establish Theorem 2 in the Introduction, which is obtained by putting together information from Proposition \ref{details} and Theorem \ref{summary-4}.  The relevant example is described below. The notation in \ref{ex} will be in effect throughout the section. 

\begin{example}
\label{ex} Let $X=\{X_1, X_2,\dots, X_5\}$ be a set of indeterminates,
 $\mathfrak c$ be the ideal of $\mathbb Z[X]$ generated by the elements:
$$
X_1^2-X_2X_3,\ X_2^2-X_3X_5,\ X_3^2-X_1X_4,\ X_4^2,\ X_5^2,\ X_3X_4,\ X_2X_5,\ X_4X_5,
$$and $A$ be the ring $\mathbb Z[X]/\mathfrak c$. 
We denote the image of the variable $X_i$ in $A$ by  $x_i$. Let $f_1$ and $f_2$ be the elements $$
f_1=x_1+x_2+x_4\quad\text{and}\quad  f_2= x_2+x_3+x_5.
$$ of A.

Fix a field $k$ and set  
$B=k\otimes_{\mathbb Z} A$.
Since $B$ is proved below to be artinian,  we can also write $B=Q/\fa$, where $Q$ is the power series ring $k[[X]]$ and 
$\fa=\mathfrak cQ$. Let $I$ be the ideal $(f_1,f_2)B$.
\end{example}

Our first calculation is made  over the ring of integers $\mathbb Z$. 
\begin{lemma}\label{over-Z} 
Let $E$ be the Koszul complex $$E:\quad 0\to A\xrightarrow{\bmatrix -f_2\\f_1\endbmatrix} A^2\xrightarrow{\bmatrix f_1&f_2\endbmatrix}A.$$ 
Then the following statements hold. \begin{enumerate}[\rm(a)]
\item The rings $A$ and $A/(f_1,f_2)$ are free ${\mathbb Z}$-modules.
\item The homology $H_{\bullet}(E)$ is free  as a module over ${\mathbb Z}$.
\item The homology  $H_{\bullet}(E)$ is free  as a module over $A/(f_1,f_2)$. Furthermore the elements    $\cls(1)$ in $H_0$, $\cls(\theta_1)$, and $\cls(\theta_2)$ in $H_1$, and $\cls(\Delta)$ in $H_2$ form a basis for $H_{\bullet}(E)$ over $A/(f_1,f_2)$, where 
\begin{align}\theta_1&=\left[\begin{matrix}
   x_1-x_2      \\-x_3+x_4+2x_5  \end{matrix}\right]\quad\text{and}\quad \theta_2=\left[\begin{matrix}
    x_4  \\x_2-x_3-x_4 \end{matrix}\right]\text{ in $E_1$ and }\notag\\\Delta&=\det\left[\begin{matrix}
   x_1-x_2     & x_4         \\
  -x_3+x_4+2x_5 &x_2-x_3-x_4
\end{matrix}\right]\text{ in $E_2$}.\notag\end{align}
\end{enumerate}
\end{lemma}

\begin{proof} (a)  One may use  Buchsberger's algorithm to check that the listed generators already  form a Gr\"obner basis. 
 When using this algorithm, there is no need to check the $S$-polynomial for a pair of monomials and there is no need to check the $S$-polynomial for two polynomials whose leading terms are relatively prime. Thus, one need only check the $S$-polynomial for the pair $X_3X_4$ and $\underline{X_3^2}-X_1X_4$ and the $S$-polynomial for the pair $X_2X_5$ and $\underline{X_2^2}-X_3X_5$. Both $S$-polynomials reduce in the appropriate manner. (We have underlined the leading terms.) Hence the listed generators are already Gr\"obner basis.  Every leading coefficient is a unit in ${\mathbb Z}$; hence $A$ is a free ${\mathbb Z}$-module.
 Furthermore, there is no difficulty using the Gr\"obner basis  to show that $x_1x_2$, $x_1x_3$, $x_1x_4$, $x_1x_5$, $x_2x_3$, $x_2x_4$, $x_3x_5$ is a basis for $A_2$, $x_1x_2x_3$, $x_1x_2x_4$, $x_1x_3x_5$  is a basis for $A_3$, and $A_4=0$.
In a similar manner,  $X_2+X_3+X_5$, $X_1-X_3+X_4-X_5$, $X_5^2$, $X_4X_5$, $X_3X_5$, $X_4^2$, $X_3X_4$, $X_3^2$ is a Gr\"obner basis for the ideal of ${\mathbb Z}[X_1,X_2,X_3,X_4,X_5]$ which defines $A/(f_1,f_2)$ and therefore $A/(f_1,f_2)$ is a free ${\mathbb Z}$-module with basis $1,x_3,x_4,x_5$.

\smallskip\noindent (b)  We treat the entire calculation as a calculation of free $\mathbb Z$-modules. We made the calculation by hand and also by using Macaulay2 \cite{M2}. The arXiv version of this paper includes an  appendix which contains  the Macaulay2 commands that we used for the computer computation, so that the reader can easily reproduce this computation. 
We have already observed that the homology $H_0(E)=A/(f_1,f_2)$ is free as a ${\mathbb Z}$-module. One computes (see the appendix on the arXiv)  that the module of $1$-cycles $Z_1(E)$ is a free ${\mathbb Z}$-module of rank $20$ and one identifies a basis for this module. Similarly one computes that  the module of $1$-boundaries $B_1(E)$ is a free ${\mathbb Z}$-module of rank $12$ and one identifies a basis for this module. By comparing the two bases, one sees    that $B_1(E)$ is a direct summand of  $Z_1(E)$ as $\mathbb Z$-modules and  
one now knows that $H_1(E)$ is a free ${\mathbb Z}$-module with basis represented by the $1$-cycles  $\theta_1,\theta_2,x_3\theta_1,x_3\theta_2,x_4\theta_1,x_4\theta_2,x_5\theta_1,x_5\theta_2$. In a similar manner, one computes that $Z_2(E)$ (which is equal to $H_2(E)$) is the free ${\mathbb Z}$-module with basis 
$\Delta,x_3\Delta,x_4\Delta,x_5\Delta$.

\smallskip\noindent (c) The argument of (b) exhibits natural surjections of $A/(f_1,f_2)$-modules: $$\xymatrix{(A/(f_1,f_2))^2\ar@{>>}[r]&H_1(E)}\quad\text{and}\quad \xymatrix{A/(f_1,f_2)\ar@{>>}[r]&H_2(E)}.$$ All four modules are free ${\mathbb Z}$-modules; therefore, rank considerations over ${\mathbb Z}$ show that the surjections are isomorphisms. 
\end{proof}

\begin{proposition}\label{details}
The following hold for the data of Example~{\rm\ref{ex}}: 
\begin{enumerate}[\quad\rm(1)]
\item $B$ is an artinian local ring with Hilbert series $H_B(z)=1+5z+7z^2+3z^3$. 
\item The algebra $B$ is Koszul.
\item The ideal $I$ is a q.c.i.\ and $H_{B/I}(z)=1+3z$. 
\end{enumerate}
\end{proposition}

\begin{proof} These assertions follow from Lemma~\ref{over-Z} because $B=k\otimes_{\mathbb Z}A$. Indeed,
%
%
%
%

each short exact sequence 
$$0\to Z_i(E)\to E_i\to B_{i-1}(E)\to 0 \quad\text{and}\quad 0\to B_i(E)\to Z_i(E)\to H_i(E)\to 0$$ from (b) is split exact over ${\mathbb Z}$ and remains split exact after the functor $k\otimes_{{\mathbb Z}}- $ has been applied. The isomorphism $H_{\bullet}(k\otimes_{\mathbb Z} E)\simeq k\otimes_{\mathbb Z} H_{\bullet}(E)$ has been established. 
The calculation of Lemma~\ref{over-Z}(c) continues to hold over $k\otimes_{\mathbb Z} A$ and therefore, $(f_1,f_2)(k\otimes_{\mathbb Z} A)$ is a q.c.i. ideal of $k\otimes_{\mathbb Z} A$.
\end{proof}

We use next the notation of \ref{Lie} regarding homotopy Lie algebras. 
We use  the recipe in \cite[Cor.~1.3]{L} (see also \cite[Ex.~10.2.2]{Avr98}) to compute the graded Lie algebra $\pi^*(B)$. This technique is explained in significant detail in \cite[{Sect.~3}]{AGP-early}. We may apply the technique because  $\operatorname{Ext}^*_{B}(k, k)$ is generated as a
$k$-algebra  in degree $1$ since $B$ is a Koszul algebra.

\begin{lemma}\label{1-dim}
$\zeta^2(B)$ is a $1$-dimensional vector space. 
\end{lemma}

\begin{proof}
Since the algebra $B$ is Koszul with Hilbert series described above, we have that the Poincar\'e series $\Po^B_k(z)$
(which is defined to be $\sum_{i=0}^\infty \dim_k\operatorname{Tor}_i^B(k,k)z^i$) is equal to
$$
\Po^B_k(z)=\frac{1}{1-5z+7z^2-3z^3}=1+5z+18z^2+58z^2+\dots...
$$
The ranks of the vector spaces $\pi^i(B)$, denoted $\varepsilon_i$ and called the {\it deviations of $B$}, may be read from this series using the techniques of \cite[Rmk.~7.1.1 and Thm.~10.2.1(2)]{Avr98}:

\begin{align*}
&\operatorname{rank}_k\pi^1(B)
=5\\
&\operatorname{rank}_k\pi^2(B)=18-\binom{5}{2}=8\\ 
&\operatorname{rank}_k\pi^3(B)=58-8\cdot 5-\binom{5}{3}=8 
\end{align*}
and so on.  At any rate, $\pi^1(B)$ has basis 
$t_1$, $t_2$, $t_3$, $t_4$, $t_5$ with the following relations: 
\begin{align*}
&[t_1,t_2]=[t_1,t_3]=[t_2,t_4]=[t_1,t_5]=0\\
&[t_2,t_3]=t_1^{(2)}\\
&[t_1,t_4]=t_3^{(2)}\\
&[t_3,t_5]=t_2^{(2)}.
\end{align*}
The following elements of $\pi^2(B)$ are then linearly independent and hence form a basis for $\pi^2(B)$:
\begin{gather*}
u_1=t_1^{(2)}\quad u_2=t_2^{(2)}\quad u_3=t_3^{(2)}\quad u_4=t_4^{(2)}\quad t_5=t_5^{(2)}\\
u_6=[t_2,t_5]\quad u_7=[t_3,t_4]\quad u_8=[t_4,t_5].
\end{gather*}
Computing the brackets $[t_i,u_j]$ and using the Jacobi identities and the relations $[t_i,t_i^{(2)}]=0$, we see that the following elements form a basis for $\pi^3(B)$: 
\begin{gather*}
v_1=[t_1,t_4^{(2)}]=-[t_4,t_3^{(2)}]=[t_3,[t_3,t_4]]\\
 v_2=[t_2,t_5^{(2)}]=-[t_5,[t_2,t_5]]\\
 v_3=[t_3,t_5^{(2)}]=-[t_5,t_2^{(2)}]=[t_2,[t_2,t_5]]\\
v_4=[t_4,t_5^{(2)}]=-[t_5,[t_4,t_5]]\\
v_5=[t_5, t_4^{(2)}]=-[t_4,[t_4,t_5]]\\
v_6=[t_4,[t_2,t_5]]=-[t_2,[t_4,t_5]]\\
 v_7=[t_3,[t_4,t_5]]=-[t_5,[t_3,t_4]]\\
v_8=[t_3,t_4^{(2)}]=-[t_4,[t_3,t_4]].
\end{gather*}
Unless listed above, all the other brackets $[t_i,u_j]$ are zero. (The signs which pertain to the Lie bracket in a graded Lie algebra may be found in \cite[Rmk.~10.1.2]{Avr98}.)

Now let us take an element $\xi$  in $\pi^2(B)$:
$$
\xi=C_1u_1+\dots +C_8u_8.
$$
If $\xi$ is a central element in $\pi^2(B)$, then we need to have $[t_i,\xi]=0$ for all $i$. For $i=5$, we have:
$$
0=[t_5,\xi]=-C_2v_3+C_4v_5-C_6v_2-C_7v_7-C_8v_4
$$
and this yields $C_2=C_4=C_6=C_7=C_8=0$. 
Then for $i=4$, we have:
$$
0=[t_4,\xi]=-C_3v_1+C_5v_4+C_6v_6-C_7v_8-C_8v_5
$$
which yields $C_3=C_5=0$. On the other hand, note that $[t_i,u_1]=0$ for all $i$. Thus $\zeta^2(B)$ is the vector space generated by $t_1^{(2)}$. 
\end{proof}

\begin{theorem}\label{summary-4}  The ideal $I$ of $B$ is a non-principal, non-embedded, minimal q.c.i.\ ideal.
\end{theorem}

\begin{proof} The proof that $I$ is a q.c.i.\ ideal in $B$ is contained in Proposition~\ref{details}. 

Apply Lemma \ref{non-emb} to see that the q.c.i.\ ideal  $I$ of $B$ is not an embedded q.c.i.\ ideal. Indeed, according to Lemma
\ref{1-dim}, we have:  $\rank_k\zeta^2(B)=1<2-0=\nu(I)-\grade_B(I)$.

It remains to show that $I$ is a minimal q.c.i.\ ideal. If $I'\subsetneqq I$ were another q.c.i.\ ideal, then 
 it follows from Proposition \ref{loewythm}(2) that 
  $\nu(I')\le 2$. We treat the cases $\nu(I')=1$ and $\nu(I')=2$ separately.

We first show that $\nu(I')=1$ is not possible; that is, we prove that $I$ does not contain any exact zero-divisors from $B$.
In light of Corollary \ref{cor exact pair}, it suffices to show that the ideal $(X_1+X_2+X_4,X_2+X_3+X_5)$ of $Q$ does not contain any homogeneous minimal generators of the ideal $\mathfrak a$ that factor non-trivially. Suppose that $a,b,c,d,e,f,g$ are elements of $k$ with the product
\begin{equation}\label{prod} [a(X_1+X_2+X_4)+b(X_2+X_3+X_5)][cX_1+dX_2+eX_3+fX_4+gX_5]\end{equation}  equal to a minimal generator of $\mathfrak a$. The ideal $\mathfrak a$ is generated by homogeneous forms of degree $2$; so the element of (\ref{prod}) is a minimal generator of $\mathfrak a$ if and only if this element is in $\mathfrak a$ and this occurs if and only if the following seven expressions vanish
\begin{equation}\label{7-gen}\begin{array}{l}ac + bd + ae + be\\ ad + bd + be + bg\\ ac + be + af\\ ac + bc + ad\\ bc + ae\\ bc + ag\\ ad + af + bf.
\end{array}\end{equation} The first expression in (\ref{7-gen}) is obtained by setting the coefficient of $X_1^2$ plus the coefficient of $X_2X_3$ in (\ref{prod}) equal to zero; the fourth expression is obtained by setting the coefficient of $X_1X_2$ in (\ref{prod}) equal to zero; and so on.  We observe that  if the seven expressions of (\ref{7-gen}) are zero, then the product (\ref{prod}) is also zero. Indeed, Macaulay2 \cite{M2} shows that in polynomial ring $\mathbb Z[a,b,c,d,e,f,g]$, the ideal $((a,b)(c,d,e,f,g))^2$ is contained in the ideal generated by the elements of (\ref{7-gen}). This inclusion of ideals passes to every field. This completes the proof that $I$ does not contain any exact zero-divisors.

Now suppose that $I'\subseteq I$ is a q.c.i.\ with $\nu (I')=2$. According to Lemma \ref{det}, or Lemma \ref{2-gen-form}, there are elements $a$, $b$, $c$, $d$ and $a'$, $b'$, $c'$, $d'$ in the maximal ideal $\mathfrak m_B$ of $B$ such that 
$(0:_BI)=\Delta B$, $(0:_B\Delta)=I$, $(0:_BI')=\Delta' B$ and $(0:_B\Delta')=I'$ with $\Delta=ad-bc$ and $\Delta'=a'd'-b'c'$. The inclusion $I'\subseteq I$ yields $I'\Delta\subset I\Delta=0$ and $\Delta\in (0:_BI')=\Delta' B$. It follows that $\Delta=\alpha \Delta'$ for some $\alpha\in B$. The element $\Delta$ is explicitly calculated in the proof of Proposition \ref{details}. This element of $B$ is homogeneous of degree two. All four elements 
$a',b',c',d'$ are in  $\mathfrak m_B$; so, $\Delta'$ is in $\mathfrak m_B^2$. The element $\Delta$ is not in $\mathfrak m_B^3$; hence $\alpha\notin \mathfrak m_B$. Thus, $\alpha$ is a unit; the ideals $\Delta B$ and $\Delta' B$ of $B$ are equal and 
$I=(0:_B \Delta)=(0:_B\Delta')=I'$.
This completes the argument that $I$ is  a  $2$-generated minimal q.c.i.\ ideal in $B$.
\end{proof}

\begin{remark}The ring $B$ in Example {\rm\ref{ex}} is an embedded deformation in the sense that 
$B=Q/\fa'\otimes_Q Q/(\theta)$ with  $\theta$ regular on $Q/\fa'$, for $\theta=X^2_1-X_2X_3$ and $\fa'$ equal to $(X_2^2-X_3X_5,\ X_3^2-X_1X_4,\ X_4^2,\ X_5^2,\ X_3X_4,\ X_2X_5,\ X_4X_5)$. (This is a Macaulay2 calculation made over the field $\mathbb Q$.) Nonetheless the q.c.i.\ ideal $I$ of $B$ is not an embedded \qci ideal. 

 On the other hand, there is an elementary argument that $B$ does not have the form $B=Q/\fa''\otimes_Q Q/(\theta_1,\theta_2)$ with  $\theta_1,\theta_2$ a  regular sequence on $Q/\fa''$. The betti numbers of $B$, as a $Q$-module are $(b_0,...,b_5)=(1, 8, 20, 23, 13, 3)$. (Again, this is a Macaulay2 calculation, made over $\mathbb Q$.)  If $B$ were equal to $Q/\fa''\otimes_Q Q/(\theta_1,\theta_2)$ with  $\theta_1,\theta_2$ a  regular sequence on $Q/\fa''$, then the betti numbers of $Q/\mathfrak a''$ would have to be $(b_0,b_1,b_2,b_3)=(1, 6, 7, 3)$.
However, the Euler characteristic forbids these numbers from being the betti numbers of a module because no module has rank equal to   $-1$.
\end{remark}

\section{Generic complete intersections of quadrics}\label{quad}

The main result of this section is Theorem \ref{5.2}, which describes when an artinian complete intersection defined by generic quadratic forms has exact zero-divisors, thereby establishing Theorem 3 in the Introduction.
Theorem \ref{5.2} is a consequence of Theorem \ref{5.1}, Proposition \ref{exact pair}, and Corollary \ref{cor exact pair} and its proof is given at the end of the section.

\begin{theorem}\label{5.2}
Let $P$ be the polynomial ring $ k[x_1,\dots,x_n]$ for some algebraically closed field $k$ of characteristic not equal to $2$ and let $A=P/(f_1,\dots,f_n)$.
\begin{enumerate}[\quad\rm(1)]
\item Assume  $n\le 4$. If $f_1,\dots,f_n$ is any regular sequence of quadratic forms in $P$, then  $A$ contains a homogeneous linear exact zero-divisor.
 \item Assume $5\le n$. If $f_1,\dots,f_n$ is a generic regular sequence of quadratic forms in $P$, then $A$ does not contain any exact zero-divisor.
\end{enumerate}
\end{theorem}

For the purposes of Theorem \ref{5.2},  a regular sequence $\pmb f=f_1, \dots, f_n$  is said to be  {\it generic} if it is an element of the open set $\mathcal I$ below. 

\begin{theorem}\label{5.1}Let $P$ be the polynomial ring $ k[x_1,\dots,x_n]$ for some algebraically closed field $k$ of characteristic not equal to $2$, and let $\mathbb A$ be the affine space $$\mathbb A=\{\pmb f=(f_1,\dots,f_n)\mid \text{ such that each $f_i$ is a quadratic form in $R$}\},$$ and $\mathcal I$ be the following subset of $\mathbb A${\rm:}
$$\mathcal I=\left\{\pmb f=(f_1,\dots,f_n)\in \mathbb A\left\vert 
\begin{array}{l}
\text{$f_1,\dots,f_n$ is a  regular sequence and every}\\
\text{non-zero element of the $k$-vector space   which}\\ 
\text{is   spanned by $f_1,\dots,f_n$  is {\bf irreducible} in $P$}\end{array}\right.\right\}.$$ Then the following statements hold.
\begin{enumerate}[\quad\rm(1)]
\item The set   $\mathcal I$ is   open in $\mathbb A$.
\item If $n\le 4$, then $\mathcal I$ is empty.
\item If $5\le n$, then $\mathcal I$ is non-empty.
\end{enumerate}
\end{theorem}
\newtheorem*{proof0}{Proof of {\rm (1)} from Theorem {\rm\ref{5.1}}}
\begin{proof0}Each $f_h$ in the definition of $\mathbb A$ is a homogeneous form in $P$ of degree $2$; consequently, the  affine space $\mathbb A$ of Theorem \ref{5.1} has dimension $n\binom{n+1}2$. The subset $\mathcal I$ of $\mathbb A$ is the complement of 
$X\cup Y$ where 
\begin{equation}\label{X} X=\left\{\pmb f=(f_1,\dots,f_n)\in \mathbb A\left\vert \begin{array}{l} \text{there exist elements $b_1,\dots,b_n$ in $ k$,}\\ \text{not all of which are zero, such that}\\ \text{ $\sum\limits_{i=1}^{n}b_if_i$ is {\bf reducible}}\end{array}\right.\right\}\end{equation}
and $$Y=\left\{\pmb f=(f_1,\dots,f_n)\in \mathbb A\left\vert \text{$f_1,\dots,f_n$ is not a regular sequence}\right.\right\}.$$ We show in Observation~\ref{Doo1}   that $Y$ is a closed subset of $\mathbb A$ and  in Observation~\ref{aug26.1} that $X$ is a closed subset of $\mathbb A$. \qed \end{proof0}

\begin{observation}\label{Doo1} Let $P=\pmb k[x_1,\dots x_n]$. Fix a sequence of degrees $\pmb d=(d_1,\dots,d_n)$. Consider  sequences of forms $\pmb f=(f_1,\dots,f_n)$ from $P$, where $f_i$ is homogeneous of degree $d_i$. Let $\mathbb A$ be the space of coefficients for $\pmb f$. Then there exists a closed set $Y\subseteq \mathbb A$ such that the coefficients of $\pmb f$ are in $Y$ if and only if $\pmb f$ is not a regular sequence.\end{observation}

\begin{proof} The polynomials of $\pmb f$ form a regular sequence if and only if the following inclusion of ideals $$(x_1,\dots,x_n)^N\subseteq (f_1,\dots,f_n)$$ holds, for $N=\sum d_i-n+1$. 
The above inclusion of ideals holds if and only if various statements about vector spaces hold; namely,

\hskip-19pt\begin{tabular}{ll}
    $\phantom{\iff{}}(x_1,\dots,x_n)^N\subseteq (f_1,\dots,f_n)$\\
$\iff(x_1,\dots,x_n)_N\subseteq (f_1,\dots,f_n)_N$\\
$\iff(x_1,\dots,x_n)_N= (f_1,\dots,f_n)_N$&$\text{since $(f_1,\dots,f_n)_N\subseteq (x_1,\dots,x_n)_N$}$\\
$\iff\dim (f_1,\dots,f_n)_N=\dim (x_1,\dots,x_n)_N$.\end{tabular}

\noindent Let $T$ be the matrix which expresses a generating set for $(f_1,\dots,f_n)_N$ in terms of the monomial  basis for $(x_1,\dots,x_n)_N$. 
The vector space $(f_1,\dots,f_n)_N$ is generated by $\{m_{N-d_j,i}f_j\}$ where, for each fixed $d$, $\{m_{d,i}\}$ is the set of monomials in $x_1,\dots,x_n$ of degree $d$. Express each $m_{N-j,i}f_j$ in terms of the basis $\{m_{N,i}\}$. We have:
$$[m_{N,1},\dots m_{N,\text{last}}]T=[m_{N-d_1,1}f_1,\dots,m_{N-d_n,\text{last}}f_n].$$
We have shown that $\pmb f$ is not a regular sequence if and only $I_{\text{row size}}(T)= 0$;  
this is a closed condition on the coefficients of $\pmb f$.  \end{proof}

\begin{observation}\label{aug26.1} Retain the notation and hypotheses of Theorem {\rm\ref{5.1}} and {\rm(\ref{X})}. Then $X$ is a closed subset of $\mathbb A$. \end{observation}

\begin{proof} The coordinate ring  for $\mathbb A$ is $S= k[\{z_{i,j;h}\mid 1\le i\le j\le n\text{ and }1\le h\le n\}]$. The point $\pmb a=(\{a_{i,j;h}))$ in affine   space $\mathbb A^{n\binom{n+1}2}$  corresponds to the element $\pmb f_{\pmb a}=(f_1,\dots,f_n)$ in $\mathbb A$ with $f_h=\sum_{i\le j}a_{i,j;h}x_ix_j$. We describe an ideal $J$ of $S$ so that every polynomial of $J$ vanishes at the point $\pmb a$  of affine   space $\mathbb A^{n\binom{n+1}2}$ if and only if $\pmb f_{\pmb a}$ is in $X$.

We work in the polynomial ring 
$$T= k[x_1,\dots,x_n, \{z_{i,j;h}\mid 1\le i\le j\le n\text{ and }1\le h\le n\}, w_1,\dots,w_n].$$
Let $\pmb F$ be the n-tuple $(F_1,\dots,F_n)$, where $F_h=\sum_{i\le j}z_{i,j;h}x_ix_j$, $F$ be the polynomial  $F=\sum_{h=1}^n F_iw_i$,   $H$   be the $n\times n$ matrix $H=(\frac{\partial^2 F}{\partial x_i
 \partial x_j})$, and $G_1,\dots,G_{\alpha}$ be a set of generators for the ideal $I_3(H)$. Each $G_\ell$ is a tri-homogeneous   polynomial in $T$ with degree $0$ in the $x$'s, degree $3$ in the $z$'s, and degree $3$ in the $w$'s. 
For each large $N$, let 
$\mu_{N,1},\dots,\mu_{N,\binom{N+n-1}N}$ be a list of the monomials in $\{w_1,\dots,w_n\}$ of degree $N$,
$M_N$ be the matrix which expresses each $\mu_{N-3,i} G_{\ell}$ (as $\mu_{N-3,i}$ roams over the monomials of degree of $N-3$ in $\{w_1,\dots,w_n\}$ and $1\le \ell\le \alpha$) in terms of the monomials $\{\mu_{N,1},\dots,\mu_{N,\binom{N+n-1}N}\}$  of degree $N$ in $\{w_1,\dots,w_n\}$:
$$[\mu_{N-3,1}G_1,\dots,\mu_{N-3,\binom{N+n-4}{N-3}}G_{\alpha}]=[\mu_{N,1},\dots,\mu_{N,\binom{N+n-1}N}]M_N.$$ Notice that each entry of each matrix $M_N$ is a cubic form in $S=k[\{z_{i,j;h}\}]$. Let  $J_N$ be the ideal in $S$ generated by the $\binom{N+n-1}N$ minors of $M_N$. Let $J$ be the ideal $\sum_NJ_N$ of $S$.

Let $\pmb a\in \mathbb A^{n\binom{n+1}2}$. We claim that $\pmb f_{\pmb a}$ is in $X$ if and only if $\pmb a\in V(J)$. Let $\pmb x$ be the variables $(x_1,\dots,x_n)$ and $\pmb w$ be the variables $(w_1,\dots,w_n)$. Observe that 
  \begin{align}\label{1.1}\text{$\pmb f_{\pmb a}$ is in $X$}& \iff \text{$\exists \pmb b\in \mathbb A^n$ with $\pmb b\neq 0$ and $F(\pmb x,\pmb a,\pmb b)$ is reducible}\\
&\iff \text{$\exists \pmb b\in \mathbb A^n$ with $\pmb b\neq 0$ and $\rank H(\pmb a,\pmb b)\le 2$} \label{1.2} \\
&\iff \text{$\exists \pmb b\in \mathbb A^n$ with $\pmb b\neq 0$ and $I_3( H(\pmb a,\pmb b))=0$}\label{1.3}\\
&\iff \left\{\begin{array}{l}\text{the ideal $I_3( H(\pmb a,\pmb w))$ of the polynomial}\\\text{ring $ k[w_1,\dots,w_n]$ is not primary to the maximal}\\\text {ideal $(w_1,\dots,w_n)$}\end{array}\right.\label{1.4}\\
&\iff \text{$(w_1,\dots,w_n)^N\not\subseteq  I_3( H(\pmb a,\pmb w))=0$, for any $N$}\label{1.5}\\
&\iff \text{every $\binom{N+n-1}{N}$ minor of $M_N(\pmb a)$ is zero for all $N$}\label{1.6}\\
&\iff \text{$\pmb a$ is in $V(J)$.}
\end{align} 

We explain the various equivalences. The point of (\ref{1.1}) is that if $\pmb f_{\pmb a}$ is the $n$-tuple $(f_1,\dots,f_n)$ in $\mathbb A$, then $F(\pmb x,\pmb a,\pmb b)$ is the element $b_1f_1+\dots+b_nf_n$ in the vector space spanned by $f_1,\dots,f_n$. Hence (\ref{1.1}) is the definition of the set $X$.

\smallskip\noindent(\ref{1.2}) The matrix $H(\pmb a,\pmb b)$ is the Hessian of the polynomial $b_1f_1+\dots+b_nf_n$ in $P=k[x_1,\dots,x_n]$. Lemma \ref{L-8-27} shows  that a quadratic form in $P$ is irreducible if and only if its Hessian has rank at least $3$.

\smallskip\noindent(\ref{1.3}) This is obvious.

\smallskip\noindent(\ref{1.4}) This is the critical translation where we are able to remove the words ``$\exists \pmb b $''. If $\mathcal S$ is a set  of homogeneous polynomials in $ k[w_1,\dots,w_n]$, with $ k$ algebraically closed, then the homogeneous Nullstellensatz  guarantees that the polynomials of $\mathcal S$ have a common non-trivial solution in $ k$  if and only if the ideal  generated by the elements of $\mathcal S$ is not primary to the irrelevant ideal $(w_1,\dots,w_n)$.  

\smallskip\noindent(\ref{1.5}) This is obvious.

\smallskip\noindent(\ref{1.6}) We turn (\ref{1.5}) into a vector space calculation. We look at our favorite basis for 
$k[w_1,\dots,w_n]_N$ and we express the elements of the subspace   $[I_3( H(\pmb a,\pmb w))]_N$  in terms of the basis for the entire space   $[(x_1,\dots,x_n)^N]_N$. The subspace is equal to the entire space if and only if the transition matrix has rank equal to the dimension of the entire vector space. We use the formulation that the subspace $[I_3( H(\pmb a,\pmb w))]_N$ is a proper subspace of $[(x_1,\dots,x_n)^N]_N$ if and only if every maximal minor of the transition matrix $M_N(\pmb a)$ is zero. \end{proof}

Lemma \ref{L-8-27} is well-known;  it can be seen, for example, by
writing $f$ in diagonal form and using \cite[11.2]{Fo}.
We include a short proof for the reader's convenience. Recall that the polynomial $f$ in $k[x_1,\dots,x_n]$, where $k$ is a field, is called {\it absolutely irreducible} if $f$ is irreducible in $\bar k[x_1,\dots,x_n]$, where $\bar k$ is the algebraic closure of $k$.
\begin{lemma} \label{L-8-27} Let $f$ be a quadratic form in the polynomial ring $P= k[x_1,\dots x_n]$, where $ k$ is a field  of characteristic not equal to $2$, and $H(f)$ be the $n\times n$ matrix with $\frac{\partial^2f}{\partial x_i\partial x_j}$ in the position row $i$ and  column $j$. Then $f$ is absolutely irreducible   if and only if $3\le \rank H(f)$.
\end{lemma}

\begin{proof} We pass to the algebraic closure $\bar k$ of $k$. Neither statement ``$f$ is absolutely irreducible'' nor ``$3\le \rank H(f)$'' is affected.  Notice that $\rank H(f)$ is invariant under change of variables. Also, the ability, or lack of ability, to factor $f$ into a product of two linear forms is invariant under change of variables. Thus, we may change variables at will. 

If $f$ factors into $\ell_1\ell_2$, then  we may change variables and assume that $f=x_1x_2$ or $f=x_1^2$. In either event, $\rank H(f)\le 2$. Now we assume that $\rank H(f)\le 2$. It follows that the vector space $(\frac{\partial f}{\partial x_1},\dots, \frac{\partial f}{\partial x_n})$ has dimension at most two; so, after a change of variables, $(\frac{\partial f}{\partial x_1},\dots, \frac{\partial f}{\partial x_n})=(x_1,x_2)$. (This is the point where we use the hypothesis that the characteristic of $k$ is not two.) It follows that $f$ is a homogeneous polynomial in two variables; hence, $f$ is reducible now that we have passed to $\bar k$.     \end{proof}
\newtheorem*{proof1}{Proof of {\rm (2)} from Theorem {\rm\ref{5.1}}}
\newtheorem*{proof2}{Proof of {\rm (3)} from Theorem {\rm\ref{5.1}}}
\newtheorem*{proof3}{Proof of Theorem {\rm\ref{5.2}}}
\begin{proof1} There is nothing to show for $n\le 2$. Fix $\pmb a\in \mathbb A^{n\binom{n+1}2}$ for $n$ equal to $3$ or $4$. We use (\ref{1.4}) to show that $f_{\pmb a}$ is in $X$. The matrix $H(\pmb a,\pmb w)$ is an $n\times n$ symmetric matrix with entries which are linear forms in the polynomial ring $k[w_1,\dots,w_n]$. Observe that 
$$\operatorname{grade}_{k[w_1,\dots,w_n]}(I_3(H(\pmb a,\pmb w))\le \begin{cases} 1<n&\text{for $n=3$}\\3<n&\text{for $n=4$; see \cite[Thm.~1]{K}.}\end{cases}$$ It follows that $I_3(H(\pmb a,\pmb w))$ is not primary to $(w_1,\dots,w_n)$;  and therefore, $\pmb f_{\pmb a}$ is in $X$.
\qed\end{proof1}
\begin{proof2} Fix $n\ge 5$. Recall that $\mathcal I=(\mathbb A\setminus X)\cup(\mathbb A\setminus Y)$ for $X$ (and $Y$) given in (and near) (\ref{X}). We know that $\mathbb A\setminus X$ is open and $\mathbb A \setminus Y$ is open and non-empty. We must show that $\mathbb A\setminus X$ is non-empty. Again, we apply (\ref{1.4}). That is, we prove the result by exhibiting an $n\times n$ symmetric matrix $W_n= (w_{ij})$ of linear forms from $k[w_1,\dots,w_n]$ such that $I_3(W_n)$ is primary to the ideal $(w_1,\dots,w_n)$. We take
$$w_{ij} =\begin{cases} w_{i+j-3}&\text{for $4\le i+j\le n+3$}\\0&\text{otherwise.}\end{cases}$$
It is obvious that each $w_i$ is in the radical of $I_3(W_n)$ for $n\ge 5$. 
\qed\end{proof2}

\begin{proof3} (1) Let $f_1,\dots,f_n$ be any regular sequence of quadratic forms from $P$ with $n\le 4$. Assertion (2) of Theorem \ref{5.1} ensures that some minimal generator of the ideal $(f_1,\dots,f_n)$ factors in a nontrivial manner in $P$. The factors represent a pair of exact zero-divisors in $A=P/(f_1,\dots,f_n)$, according to Proposition \ref{exact pair}. 

\noindent(2) Let $\pmb f=(f_1,\dots,f_n)$, with $5\le n$, be an element of the dense open subset $\mathcal I$ of $\mathbb P$, as described in Theorem \ref{5.1}. The definition of $\mathcal I$ ensures that $\pmb f$ is a regular sequence and that every minimal generator of the ideal $(f_1,\dots,f_n)$ is irreducible in $P$. The ring $A=P/(f_1,\dots,f_n)$ is artinian (hence complete) and we may apply Corollary \ref{cor exact pair} to conclude that every pair of exact zero-divisors in $A$ gives rise to a non-trivial factorization in $P$ of a minimal generator of the ideal $(f_1,\dots,f_n)$. No such factorizations exist in $P$; consequently, no exact zero-divisors exist in $A$. 
\qed\end{proof3} 

\begin{bfRemark} In \cite[3.1]{KST} a local ring $(R,\mathfrak m, k)$ is called {\it exact} if $\mathfrak m = (x_1,\dots, x_n)$ where $x_i$ is
an exact zero-divisor in $R/(x_1,\dots,x_{i-1})$ for $i = 1,\dots,n$. The resolution of $k$ given
by \cite[1.8]{KST} then yields $\operatorname{cx}_R(k) = n$, so $R$ is a complete intersection by \cite[2.3]{G}.
This sharpens one implication in \cite[2.3]{KST}.

In this terminology, (1) in Theorem \ref{5.2} becomes the statement that artinian
complete intersections of $n$ quadrics are exact when $n \le 4$.
\end{bfRemark}

\section*{appendix}

This appendix describes a Macaulay2 computation that has been used to verify the claims in the proof of Lemma \ref{over-Z}(b).
The sentences starting with {\tiny $--$} represent our comments and the sentences starting with  $>$ are the actual Macaulay2 commands. Upon running the Macaulay2 program, the reader can confirm that all our anticipated outcomes are indeed correct.

{\small \begin{verbatim}
-- Part 1 of this Macaulay2 calculation  verifies that, as modules over 
   the ring of integers, Z1 is the direct sum of B1 plus the module spanned 
   by the 8 column vectors theta1*(1,x3,x4,x5), theta2*(1,x3,x4,x5).
-- Part 2 of this calculation  verifies that, as a module over the ring
   of integers, Z2 is the free module with basis 
   the (determinant of [theta1 theta2]) *(1,x3,x4,x5)

-- We now perform Part 1.  We verify that, as modules over 
   the ring of integers,
   Z1 is the direct sum of B1 plus the module spanned by 
   the 8 column vectors 
   theta1*(1,x3,x4,x5), theta2*(1,x3,x4,x5).
-- We refer to B1 plus the module spanned by the 8 column 
   vectors as Candidate. 
-- Each module lives in 3 degrees: 2,3,4. We compare the bases in
   each degree by identifying the change of basis matrix
-- ChangeOfBasisMatrixInDegreeX which satisfies 
-- BasisForZ1InDegreeX*ChangeOfBasisMatrixInDegreeX 
   = BasisForCandidateInDegreeX
   (where X is 2, 3, or 4).
-- We observe that each of the  three ChangeOfBasis matrices 
   has a determinant which is a unit in the ring of integers.

-- We now set up the objects of interest.

> A=ZZ[x1,x2,x3,x4,x5]/ideal(x1^2-x2*x3,x2^2-x3*x5,
  x3^2-x1*x4,x4^2,x5^2,x3*x4,x2*x5,x4*x5)
> f1=x1+x2+x4
> f2= x2+x3+x5
> d1= matrix {{f1,f2}}
> d2=map(source(d1),,matrix{{-f2},{f1}})
> BasisForZ1=super basis (kernel d1)

-- Notice that the basis for Z1 has 20 generators: 3 
   of degree 2, 11 of degree 3 and 6 of degree 4.

> BasisForB1=super basis (image d2)

-- Notice that the basis for B1 has 12 generators: 1 
   of degree 2, 5 of degree 3 and 6 of degree 4.
-- We create the two column vectors theta1 and theta2:

> theta1=matrix{{x1-x2},{-x3 + x4 + 2*x5}}
> theta2=matrix{{x4},{x2 - x3 - x4}}

-- We get Macaulay2 to organize the bases by degree:

> BasisForB1InDegree2=super basis (2,image d2)
> BasisForB1InDegree3=super basis (3,image d2)
> BasisForB1InDegree4=super basis (4,image d2)
> BasisForZ1InDegree2=super basis (2,kernel d1)
> BasisForZ1InDegree3=super basis (3,kernel d1)
> BasisForZ1InDegree4=super basis (4,kernel d1)
 
-- We create a matrix whose columns are our proposed basis for Candidate
    in degree 2. Then we identify the ChangeOfBasisMatrixInDegree2.

> BasisForCandidateInDegree2=BasisForB1InDegree2|theta1|theta2
> ChangeOfBasisMatrixInDegree2
  =BasisForCandidateInDegree2//BasisForZ1InDegree2

> BasisForZ1InDegree2*ChangeOfBasisMatrixInDegree2 
  == BasisForCandidateInDegree2

-- We expect the answer true at this point.

> determinant  ChangeOfBasisMatrixInDegree2

-- Notice that the determinant of ChangeOfBasisMatrixInDegree2 is a unit.
-- We repeat the process in degree 3. Observe how we create our
   basis for Candidate in degree 3. We take Macaulay's basis 
   for B1 and then we adjoin the six column vectors theta*x. 

> BasisForCandidateInDegree3=
       BasisForB1InDegree3|theta1*matrix{{x3,x4,x5}}
       |theta2*matrix{{x3,x4,x5}}
> ChangeOfBasisMatrixInDegree3
   =BasisForCandidateInDegree3//BasisForZ1InDegree3

> BasisForZ1InDegree3*ChangeOfBasisMatrixInDegree3 
   == BasisForCandidateInDegree3

-- We expect the answer true at this point.

> determinant ChangeOfBasisMatrixInDegree3

-- Notice that the determinant of ChangeOfBasisMatrixInDegree3 is a unit.
-- We repeat the process in degree 4. This step is probably redundant. It
   is apparent to the naked eye that B1 in degree 4 already equals 
   the module of all column vectors with 2 entries from  A in degree 3.  

> BasisForCandidateInDegree4=BasisForB1InDegree4
> ChangeOfBasisMatrixInDegree4
   =BasisForCandidateInDegree4//BasisForZ1InDegree4

> BasisForZ1InDegree4*ChangeOfBasisMatrixInDegree4 
   == BasisForCandidateInDegree4

-- Again, we hope for true.

> determinant ChangeOfBasisMatrixInDegree4 

-- Again, the determinant should be a unit.

-- This completes Part 1. Now we attack Part 2, where we study Z2.
   We verify that, as a module over the ring of integers, Z2 is 
   the free module with basis 
   the (determinant of [theta1 theta2]) *(1,x3,x4,x5).

> BasisForZ2=super basis (kernel d2)

-- Notice that the basis for Z2 has 4 generators: 1 of degree 4,
   3 of degree 5.
-- Observe that the generator of Z2 of low degree is equal to 
   the determinant of [theta1 theta2] (which we call Delta).

> BasisForZ2_(0,0)
> Delta = determinant (theta1|theta2)
> BasisForZ2_(0,0)==Delta

-- We anticipate the answer true.
-- It is clear that the three entries in the row vector Delta*(x3,x4,x5) 
   generate all of A in degree 3.
-- It is also clear that the three basis elements of Z2 of high degree
   also generate all of A in degree 3.
-- At any rate, the two matrices are equal on the nose:

> submatrix'(BasisForZ2,,{0})
> Delta*matrix{{x3,x4,x5}}
> BasisForZ2_(0,1) == Delta*x3
> BasisForZ2_(0,2) == Delta*x4
> BasisForZ2_(0,3) == Delta*x5

-- We anticipate the answer true in each case.
\end{verbatim}
}
\begin{Acknowledgment} The referee read our paper quickly, carefully, and thoughtfully. We appreciate the numerous suggestions that lead to improvements in our exposition and in our proofs. 
\end{Acknowledgment}

  \end{document}